\documentclass[reqno,a4paper,12pt]{amsart}
\usepackage{latexsym, amsmath, amssymb, amsthm, a4}

\usepackage{amsfonts, bbold, bbm, marvosym}

\usepackage{mathrsfs}

\newtheorem{theorem}{Theorem}[section]

\newtheorem{lemma}[theorem]{Lemma}

\theoremstyle{remark}
\newtheorem{remark}[theorem]{Remark}

\renewcommand{\a}{\alpha}
\renewcommand{\b}{\beta}
\newcommand{\g}{\gamma}
\newcommand{\G}{\Gamma}
\renewcommand{\d}{\delta}

\newcommand{\e}{\epsilon}

\newcommand{\z}{\zeta}
\newcommand{\y}{\eta}

\renewcommand{\i}{\iota}

\renewcommand{\l}{\lambda}
\renewcommand{\L}{\Lambda}
\newcommand{\m}{\mu}
\newcommand{\n}{\nu}
\newcommand{\x}{\xi}

\newcommand{\s}{\sigma}

\newcommand{\vs}{\varsigma}
\renewcommand{\t}{\tau}

\renewcommand{\o}{\omega}

\newcommand{\R}{{\mathbb R}}

\newcommand{\N}{{\mathbb{N}}}

\newcommand{\nup}{\pmb{\pmb{\n}}}
\newcommand{\bx}{\pmb{\xi}}
\newcommand{\bl}{\pmb{\l}}

\newcommand{\kb}{{\mathbf k}}
\newcommand{\lb}{{\mathbf l}}

\newcommand{\pb}{{\mathbf p}}

\newcommand{\rb}{{\mathbf r}}

\newcommand{\ub}{{\mathbf u}}

\newcommand{\xb}{{\mathbf x}}
\newcommand{\yb}{{\mathbf y}}
\newcommand{\zb}{{\mathbf z}}

\newcommand{\Db}{{\mathbf D}}
\newcommand{\Eb}{{\mathbf E}}

\newcommand{\Hb}{{\mathbf H}}
\newcommand{\Ib}{{\mathbf I}}

\newcommand{\Lb}{{\mathbf L}}

\newcommand{\Nb}{{\mathbf N}}
\newcommand{\Ob}{{\mathbf O}}

\newcommand{\Qb}{{\mathbf Q}}

\newcommand{\aF}{\mathfrak a}
\newcommand{\AF}{\mathfrak A}
\newcommand{\bF}{\mathfrak b}
\newcommand{\BF}{\mathfrak B}

\newcommand{\DF}{\mathfrak D}

\newcommand{\jF}{\mathfrak j}

\newcommand{\kF}{\mathfrak k}
\newcommand{\KF}{\mathfrak K}

\newcommand{\mF}{\mathfrak m}
\newcommand{\MF}{\mathfrak M}

\newcommand{\qF}{\mathfrak q}
\newcommand{\QF}{\mathfrak Q}

\newcommand{\rF}{\mathfrak r}
\newcommand{\RF}{\mathfrak R}
\newcommand{\ssF}{\mathfrak s}
\newcommand{\SF}{\mathfrak S}

\newcommand{\uF}{\mathfrak u}

\newcommand{\wF}{\mathfrak w}

\newcommand{\yF}{\mathfrak y}

\newcommand{\ZF}{\mathfrak Z}
\newcommand{\zF}{\mathfrak z}

\newcommand{\vF}{\mathfrak v}

\newcommand{\Dc}{{\mathcal D}}
\newcommand{\Ec}{{\mathcal E}}
\newcommand{\Fc}{{\mathcal F}}

\newcommand{\Kc}{{\mathcal K}}
\newcommand{\Lc}{{\mathcal L}}

\newcommand{\Qc}{{\mathcal Q}}
\newcommand{\Rc}{{\mathcal R}}

\newcommand{\dist}{{\rm dist}\,}

\newcommand{\grad}{{\rm grad}\,}

\renewcommand{\div}{{\rm div}\,}

\newcommand{\sub}{\footnotesize{\rm{sub}}\,}

\newcommand{\pd}{\partial} 

\newcommand{\Tr}{\operatorname{Tr\,}}

\newcommand{\pn}{{\partial_{\nup}}}
\newcommand{\pnu}{{\partial_{\nup}}}

\theoremstyle{plain}

\newtheorem{proposition}[theorem]{Proposition}
\newtheorem*{theorem*}{Theorem}

\theoremstyle{definition}

\renewcommand{\div}{\mbox{div}\,}

\newcommand{\la}{\langle}
\newcommand{\ra}{\rangle}

\newcommand{\beq}{\begin{equation}}
\newcommand{\eeq}{\end{equation}}

\numberwithin{equation}{section}
\numberwithin{figure}{section}

\begin{document}

\title[Polynomially compact operators]{Eigenvalue asymptotics for polynomially compact pseudodifferential operators and applications}

\author{Grigori Rozenblum }

\address{ Chalmers University of Technology and The University of Gothenburg (Sweden); St.Petersburg State University, Dept. Math. Physics (St.Petersburg, Russia)}

\email{grigori@chalmers.se}

\dedicatory{To Vassily Mikhalovich Babich, with admiration}
\thanks{The author was supported by grant RScF 20-11-20032}
\begin{abstract}
We find the asymptotics of eigenvalues of polynomially compact zero order pseudodifferential operators, the motivating example being the Neumann-Poincare operator in linear elasticity.
\end{abstract}
\maketitle



\section{Introduction}
Rather recently, a certain interest has arisen towards the spectral theory of zero order pseudodifferential operators. While the location of the essential spectrum follows easily from  basic theorems, finer spectral properties remained mostly unresolved. The first publication devoted to the topic was, probably, the paper \cite{Adams} by M.Adams in 1983, where the structure of the spectral function for a self-adjoint zero-order operator was studied. It was followed by the paper by D.Yafaev \cite{Yaf}, where zero order pseudodifferential operators had appeared in connection with the study of the quantum scattering for systems involving the magnetic field. The situation changed just recently, when in the papers \cite{CdV1}, \cite{CdV2}, \cite{Zwor}, \cite{Wang}, \cite{Tao}, inspired by some problems concerning forced waves in the stratified media, a series of important results about the properties of the  essential spectrum and scattering characteristics for zero order operators were obtained. The present paper is devoted to the study of the \emph{discrete spectrum} of a certain class of such operators, namely, the polynomially compact ones. Such pseudodifferential operators had arisen when considering the Neumann-Poincar\'e (NP) integral operators in 2- and 3 - dimensional elasticity and some other systems of differential equations. In their turn, the studies of spectral properties of the NP operators were inspired by the needs of the analysis of the plasmonic resonance in metamaterials, artificial materials possessing physical characteristics of unusual sign, see \cite{Ammari}, .

The NP operator for the electrostatic problem, a.k.a.  'the double layer potential', for a 3D body \emph{with smooth boundary}, is a compact operator, more exactly, an order $-1$ pseudodifferential operator. Although not self-adjoint in the usual $L^2$ space, this operator is self-adjoint in the Sobolev space $H^{-\frac12}$ equipped with the  norm defined by means of the single layer potential. This enables one to adapt in the 3D case the classical results by M.Birman and M.Solomyak (see \cite{BS}, \cite{BS obzor} )  and obtain  asymptotic formulas for  singular values and eigenvalues (see \cite{M}, \cite{MR}). Note that in the 2D case, the NP operator has order $-\infty$, and its (exponential) eigenvalue asymptotics is known only for ellipses. The situation with boundary of a finite smoothness  (Lipschitz or better, otherwise the NP operator is not compact) is less understood. In 3D, the asymptotics of singular values and of moduli of eigenvalues  (these are not the same since the NP operator is not self-adjoint) is found in \cite{M} for a $C^{2,\a}$ boundary. In 2D, only estimates are known, with order depending on the smoothness of the boundary, but it is still known  nothing about the sharpness of these estimates (it is a challenge for researches, to find just a single example of a 2D domain, different from an ellipse, where the asymptotics of NP eigenvalues can be found.)

  Unlike the electrostatic case, the elasticity NP operator is never compact. This fact was discovered long ago by specialists in the Elasticity Theory (see, e.g., \cite{Kupr79}); the pseudodifferential representation of the NP operator was first found, probably, in \cite{AgrLame}.  For the case of a homogeneous media,  the NP operator is polynomially compact, with essential spectrum consisting of three points in dimension 3 (the corresponding calculations can be found in \cite{3D}) and two points in dimension 2 (see \cite{IOP}, \cite{Ando2D}), where this was shown even for surfaces of a certain finite smoothness. Little bit later, it was found that for a nonhomogeneous media, in dimension 3, the polynomial compactness property may disappear; as a result, several, an even number,  intervals  of the essential spectrum can be present  (some of them may degenerate to  single points) with, additionally, one isolated point, zero, in the essential spectrum (see \cite{MR3D}). Also in \cite{MR3D}, some estimates for the eigenvalues of the NP operator were found; in particular, they show that the rate of convergence of the eigenvalues to the tips of the essential spectrum depends on the structure of these tips. Finally, in \cite{AKM},  for the NP operator in dimension 2, the dependence of the eigenvalues convergence rate   to the points of the essential spectrum on the smoothness of the boundary was investigated, again on the level of upper estimates.

In the present paper  we consider the general question on the asymptotics of eigenvalues of zero order pseudodifferential operators, of which the NP operator is a particular case, as these eigenvalues approach the tips of the essential spectrum. The paper deals with the case of a polynomially compact operator, or, equivalently, the one whose essential spectrum consists of several isolated points only. We find the leading term in the asymptotics and determine, in particular, how its order and the corresponding coefficient depend on the symbol of the operator under consideration. Our general results are applied to the 3D elasticity NP operator $\KF$.  Previously,  results on the eigenvalue asymptotics were known here  for a single explicitly calculable case only, see  \cite{DLL}, where the NP operator in the spherical geometry was considered.  Being applied to the NP operator on a smooth surface in $\R^3$, our general results produce the correct order of the asymptotics. The coefficients in the asymptotic formulas, due to our general considerations, depend on lower order symbols of  the double layer potential. These symbols, in their turn, may depend on geometric characteristics of the surface  (in fact, on its principal curvatures) and the Lame constants of the media. Finding an explicit expression is very cumbersome and  will be performed in a further publication.

The knowledge of eigenvalues of the operator $\KF$ is important in the study of the plasmonic resonance and other effects arising in the presence
   of metamaterials, the ones with unusual sign of the characteristical constants of the material, see \cite{DLL}, \cite{Ammari}, \cite{Ando2D},  \cite{3D}, \cite{IOP}, etc., and an extensive literature cited there. In particular, in the recent paper \cite{DLL}, the spectrum of the NP operator on a sphere has been explicitly found. It is established there that for each of the points of the essential spectrum, $0,\pm \mathbbm{k}$, $\mathbbm{k}=\frac{\m}{2(2\m+\l)}$, there exists a sequence of eigenvalues of $\KF$ converging to this point,
  \begin{gather}\label{ball}
\bl_k^{0}(\KF)=\frac{3}{2(2k+1)}\sim \frac{3}{4k}, \\\nonumber
\bl_k^{-}(\KF)=\frac{3\lambda-2\mu(2k^2-2k-3)}{2(\lambda+2\mu)(4k^2-1)}\sim-\mathbbm{k}+ \frac{2\mu}{(\l+2\m)k} , \\\nonumber
\bl_k^{+}(\KF)=\frac{-3\lambda+2\mu(2k^2+2k-3)}{2(\lambda+2\mu)(4k^2-1)},\sim\mathbbm{k}+ \frac{2\mu}{(\l+2\m)k} ,
\end{gather}
where $\l,\m$ are Lam\'e constants.
We can see that all three sequences approach their limit points from above; there are no eigenvalues that approach these points from below. Moreover, the sequence tending to $0$ does not depend on the material characteristics $\l,\m$. In the paper we investigate whether these properties remain for general
smooth homogeneous elastic bodies. It turns out that if the body is not convex, there  may arise sequences that approach the eigenvalues from below, with the same order of asymptotics.
Eigenvalues converging from above exist always. These latter properties are similar to the ones established in \cite{MR} for the electrostatic NP operator.

\section{General setting}
Let $\AF$ be a polynomially compact self-adjoint classical pseudodifferential operator of order zero, acting in the space of smooth sections of a Hermitian vector bundle $\Ec$ of dimension $\Nb<\infty$ over a closed connected  smooth Riemannian manifold $\G$ of dimension $d>1$. We consider the case $\Nb>1$; if $\Nb=1$, $d>1$, the problem is easily reduced to the case of a compact negative order operator. The symbol of $\AF$ has a standard asymptotic expansion in fixed local coordinates and frame, $\aF_0(x,\x)+\aF_{-1}(x,\x)+\dots$, with $\aF_0(x,\x), \aF_{-1}(x,\x)$ being the order zero and order $-1$ positively homogeneous in $\x$ terms in the symbol of $\AF$. The \emph{principal symbol} $\aF_0(x,\x), \x\ne0,$ is a Hermitian endomorphism of the fiber $\Ec_{x}$, smooth in $x,\x$ variables and zero order positively homogeneous in $\x$ variable. It is invariant in the proper sense under the change of the local coordinates. The term $\aF_{-1}$ is, generally,  not invariant.  We denote by $\aF_{\sub}$ the \emph{subprincipal symbol} of $\AF$ defined by
\begin{equation}\label{Subprin}
    \aF_{\sub}(x,\x)=\aF_{-1}(x,\x)+\frac{1}{2i}\partial_{x}\partial_{\x} \aF_0(x,\x).
\end{equation}
The subprincipal symbol is invariant with respect to the choice of local co-ordinate systems on $T^*\G$, provided
 the operator $\AF$ is considered on sections of the bundle of half-densities on $\G.$

 We denote by $\vs_j(x,\x)$, $(x,\x)\in \dot{T}^*\G(\equiv T^*\G\setminus 0)$, $j=1,\dots,\Nb,$ the eigenvalues of the symbol $\aF_0(x,\x)$, numbered in the nondecreasing order, \emph{counting multiplicities}. These eigenvalues are continuous, but not necessarily smooth, functions on $\dot{T}^*\G$ (one should not expect that the corresponding eigenvectors are even continuous). In the present paper we consider   the case of  operator $\AF$ being \emph{polynomially compact}; this means that for some polynomial $\pb(\o)$, the operator $\pb(\AF)$ is  compact. This property is determined by the principal symbol only. Namely, this can happen if and  only if the principal, zero order, symbol $\pb(\aF_0(x,\x))$ of $\pb(\AF)$ equals zero and, since the eigenvalues of $\pb(\aF_0(x,\x))$ are $\pb(\vs_j(x,\x))$, these eigenvalues $\vs_j(x,\x)$ must coincide with some zeros $\o_\i$ of the polynomial $\pb$. By this reason, the eigenvalues $\vs_j(x,\x)$, being continuous, must be constant, $\vs_j(x,\x)=\vs_j$, $\vs_j\in \pb^{-1}(0).$ We choose the polynomial $\pb$ of smallest possible degree, so we can suppose that each zero of $\pb(\o)$ is an eigenvalue of $\aF_0$ and is a simple zero of this polynomial. Finally, we suppose that the leading coefficient of $\pb(\o)$ equals 1. So, the polynomial factorizes as
\begin{equation}\label{polynom}
    \pb(\o)=\prod_{\i=1}^\Lb(\o-\o_\i),
\end{equation}
where $\Lb$ is the degree of $\pb(\o)$, $\Lb\le \Nb.$ Thus, the set $\Ob\equiv \cup\{\o_\i\}= \pb^{-1}(0)$ consists of $\Lb$ distinct eigenvalues of the principal
symbol $\aF_0$, i.e.,  of the numbers $\vs_j$, but \emph{not} counting multiplicities.
The principal symbol of $\AF$ is therefore an endomorphism with  eigenvalues $\o_\i$
 not depending on $(x,\x)$; it follows that the essential spectrum of $\AF$ coincides with
  the set $\Ob$ (a simple reasoning justifying this statement can be found, e.g., in \cite{MR3D};
 it is an easy generalization of the one for the scalar case, see \cite{Adams}.) There may also exist some finite multiplicity
  eigenvalues of $\AF$, having the points $\o_\i$ as their only possible limit points. Our aim is to find estimates and asymptotics
  of these eigenvalues, as they approach $\o_\i$.

Just a minor modification is needed if the manifold $\G$ has dimension $1$, i.e., is a circle. Here, the cotangent bundle, with the zero section removed,
 is not connected,
consequently, the eigenvalues of  the principal symbol $\aF_0(x,\x)$ may be different for $\x<0$ and for $\x>0$. Therefore,
 the number of different
  eigenvalues of the principal symbol may be twice as large as the dimension $\Nb$ of the bundle $\Ec$, $\Lb\le 2\Nb$.
  This circumstance leads to some
   obvious changes in the reasoning and notation. Here, however, even the case $\Nb=1$ is nontrivial.

Further on, we need some more  notation. Let $\BF$ be a bounded self-adjoint operator and let $\z$ be a
fixed point in the essential spectrum of  $\BF$.
Fix some points $\z_{\pm},$ $\z_-<\z<\z_+,$ so that the interval $[\z_-,\z_+]$ contains no points of
the essential spectrum of $\BF$ with  exception of $\z$. For $\t\in(\z_-,\z)$, we denote by $n_-(\BF;\z,\t)$
 the number of eigenvalues of $\BF$ in the interval $(\z_-,\t)$, counting multiplicities. Similarly,
 for $\t\in(\z,\z_+)$, the notation $n_+(\BF;\z,\t)$ stands for the number of eigenvalues of $\BF$ in $(\t,\z_+)$.
  Of course, these quantities depend on our voluntary choice of the points $\z_{\pm}$, but, if the spectrum of $\BF$ in
   $(\z_-,\z)$ or  $(\z,\z_+)$ is infinite, the leading (power) term in the asymptotics of the corresponding
   counting functions $n_\pm(\BF;\z,\t)$  as $\t\to\z\pm 0$ does not depend on this choice. For the case
   of our special interest, $\z=0$, we omit $\z$ in the notation of the above counting functions and
   write simply $n_{\pm}(\BF;\t)=n_{\pm}(\BF;0,\t)$, $\pm\t>0$. Along with the above counting functions,
   it is convenient to consider the eigenvalues themselves. With, again, $\z,\z_{\pm}$ fixed as above,
    we  denote by $\bl^{\pm}_k(\BF;\z)$ the eigenvalues of $\BF$ in the interval $(\z_-,\z)$ in the
    nondecreasing order (for $-$ sign), respectively, the eigenvalues of $\BF$ in $(\z,\z_+)$ in
    the nonincreasing order (for $+$ sign), counting multiplicities. Again, for $\z=0$, we omit $\z$ in these notations.  In the standard way, power-order
     asymptotic formulas and estimates for the eigenvalues are re-calculated to formulas for the counting function and vice versa.

     The same notations are used for the case when the operator under study is not necessarily self-adjoint but still has real spectrum.

Our basic 'abstract' result is the following.
\begin{theorem}\label{MainTheorem}
Let $\o_\i$ be a zero of the polynomial $\pb(\o)$. Then the counting functions $n_-(\AF;\o_{\i}, \t)$, $n_+(\AF;\o_{\i},\t)$ have asymptotics
\begin{equation}\label{As1}
    n_{\pm}(\AF;\o_\i, \t)\sim C_{\pm}(\o_\i)|\t-\o_\i|^{-{d}}, \, \t\to\o_{\i}\pm 0.
\end{equation}
The coefficients $C_{\pm}(\o_\i)$ are expressed via the symbols $\aF_0$ and $\aF_{-1}$; the expressions will be derived further on.
\end{theorem}

It may happen that the coefficients $C_{\pm}(\o_{\i})$ in \eqref{As1} turn out to be zero for one or several of the
points $\o_{\i}$ in the essential spectrum of $\AF$. In this, \emph{degenerate}, case, the construction proving Theorem \ref{MainTheorem}
 enables us to establish for eigenvalues near these 'special' points $\o_{\i}$ asymptotic formulas of the type \eqref{As1}, but
  with the order $|\t-\o_{\i}|^{-\frac{d}{2}}$ instead of $|\t-\o_{\i}|^{-{d}}$; if the coefficients in this latter formula turn out
   to be zero again,  the next  order asymptotics of order $|\t-\o_{\i}|^{-\frac{d}{3}}$ can be found, and so on. The corresponding formulas for coefficients become
   rather cumbersome, however a general method for obtaining them is fairly transparent. We discuss this topic
    in the proper section.

\section{Polynomially compact operators. Eigenvalue estimates near the points $\o_\i$. }

We start with obtaining an upper estimate for $n_{\pm}(\AF;\o_\i, \t)$, i.e., for the  eigenvalue counting function near the
 points of the essential spectrum of a polynomially compact self-adjoint operator  $\AF$.
\begin{proposition}\label{prop.p(A)} For any $\o_\i\in\Ob=\pb^{-1}(0)$,  the estimate
\begin{equation}\label{roughEstimate}
   n_{\pm}(\AF;\o_\i, \t) =O(|\o_\i-\t|^{-d}),
\end{equation}
holds for $\t\to \o_\i\pm 0$.
\end{proposition}
\begin{proof}
Consider the operator $\BF=\pb(\AF)$. This is a classical self-adjoint pseudodifferential operator of order zero, with principal symbol $\pb(\aF_0)$.
 We recall, however, that $\pb(\aF_0)=0$. This means that $\BF$ is, in fact, a pseudodifferential operator of order $-1$ and its 'principal symbol' of order $0$ vanishes. As an  order $-1$
 operator, $\BF$ has a principal symbol $\bF_{-1}(x,\x)$ which is invariant under the change of local co-ordinates; we do not care  for the explicit expression of this symbol at the moment.  We,
  however, know estimates and asymptotics for eigenvalues of $\BF$, following, e.g., \cite{BS}:
\begin{gather}\label{asymp.p(A)}
    n_{\pm}(\BF;\t)\sim C_{\pm}(\BF)|\t|^{-d}, \t\to \pm 0.
\end{gather}
The coefficients $C_{\pm}(\BF)$ can be expressed explicitly via the (now!) principal symbol, of  order $-1$, of $\BF$.

By the spectral mapping theorem, the eigenvalues of $\BF=\pb(\AF)$ in $\R^1\setminus[-\t,\t], \, \t>0,$ are exactly the images of the eigenvalues
 of $\AF$ in $\pb^{-1}(\R^1\setminus[-\t,\t])$ under the action of $\pb$. By our construction, since $\pb(\o)$ is a polynomial with simple
  zeros only, the derivative of $\pb(\o)$ is separated from zero on some neighborhood of the set of points $\o_\i$:
\begin{equation}\label{p'}
    |\pb'(\o)|>\e_0\,  \mathrm{for}\,  \dist(\o,\Ob)<\d_0,
\end{equation}
(with $\d_0$ chosen so small that $\d_0-$ neighborhoods of different points $\o_\i$ do not intersect.)
Since the points $\o_\i$ are the only possible limit points of the spectrum of $\AF$, there may exist only finitely many, say, $n_0$, eigenvalues of $\AF$ outside the $\d_0-$neighborhood of $\Ob$. As for the eigenvalues \emph{in} this neighborhood, it follows from \eqref{p'} that
\begin{equation}\label{distance}
    |\bl-\o_\i|\le \e_0^{-1} |\pb(\bl)-\pb(\o_\i)|=\e_0^{-1}|p(\bl)|
\end{equation}
for any eigenvalue $\bl$ of $\AF$ in the $\d_0-$neighborhood of the point $\o_\i$. This means that for $\t>0$, the number of eigenvalues
 of $\AF$ outside the $\t$-neighborhood of the set $\Ob$ is majorated by $n_0$ plus the number of eigenvalues of $\BF$ \emph{outside} the $\e_0 \t$-neighborhood of $0$, and by \eqref{asymp.p(A)}, we have:
\begin{equation}\label{estimate}
    \sum_\i[(n_+(\AF; \o_\i, \o_\i+\t)+n_-(\AF;\o_\i,\o_\i-\t))]\le n_0+n_+(\BF;\e_0 \t)+n_-(\BF;\e_0  \t)\le n_0+C \t^{-d}.
\end{equation}
It follows from \eqref{estimate} that \emph{each} term in brackets in the sum on the left-hand side is majorated by $n_0+Ct^{-d},$ as $t\to 0$.
\end{proof}

 The reasoning in the above proof establishes the asymptotics of the \emph{sum} of the counting functions for the
  eigenvalues of the operator $\AF$ in two-sided neighborhoods of \emph{all} points of the essential spectrum. What we need, however, is their \emph{separate} asymptotics.

\section{Spectral localization. The nondegenerate case}\label{nondeg}
Now we pass to the study of the asymptotic behavior of eigenvalues near the points $\o_\i$, $\i=1,\dots,\Lb$. To do it, we modify the reasoning
 in Proposition \ref{roughEstimate}. 


Consider, for a fixed $\i$, the compact operator
\begin{equation}\label{polynB}
     \BF_\i\equiv\BF_{\i,-1}=\pb_\i(\AF),
\end{equation}
where $\pb_\i(\o)$ is the polynomial of degree $2\Lb-1$,
\begin{equation}\label{polyn}
    \pb_\i(\o)=\pb(\o)^2/(\o-\o_\i)\equiv(\o-\o_\i)\prod_{l\ne\i}(\o-\o_l)^2 .
    \end{equation}
    We drop temporarily the subscript $-1$ in the notation of $\BF_{\i,-1}$.
\begin{lemma}\label{polyn.lem}
The limits $\lim_{\t\to +0} \t^{d}n_{\pm}(\AF; \o_\i, \o_\i\pm \t)$ exist, $\i=1,\dots,L$, moreover
\begin{equation}\label{B_1}
    \lim_{\t\to +0} \t^{d}n_{\pm}(\AF; \o_\i, \o_\i\pm \t)= \prod_{l\ne\i} (\o_l-\o_\i)^{2d} \lim_{\t\to+0}\t^d n_{\pm}(\BF_\i,\t),
\end{equation}
under the assumption that the limit on the right exists.
\end{lemma}
\begin{proof}  The nonzero eigenvalues of $\BF_\i$, by the spectral mapping property, are the values of the polynomial  $\pb_\i(\bl)$, for all
 eigenvalues $\bl\notin \Ob$ of the operator $\AF$, taking into account their multiplicities; additionally,  zero may
  be an eigenvalue as well, provided one of the points $\o_l$ is an eigenvalue of $\AF.$ With possible  exception of some finite set,
  the eigenvalues of $\AF$ split into sequences $\bl^{\pm,l}_k$, $l=1,\dots,L$, converging to $\o_l$ from below $(-)$ or from above
   $(+)$ as $k\to\infty$ (some, or even all,  of these sequences may be finite or even void.)
 By Proposition \ref{prop.p(A)}, for $l\ne \i,$ we have
 \begin{equation}\label{l ne 1}
    \bl^{\pm,l}_k-\o_l = O(k^{-\frac1d}).
 \end{equation}
 It follows that
\begin{equation}\label{l ne 1.1}
    \pb_\i(\bl^{\pm,l}_k)=(\bl^{\pm,l}_k-\o_\i)\prod_{j\ne\i}(\bl^{\pm,j}_k-\o_l)^2 = O(k^{-\frac2d}).
\end{equation}
In fact, for $l\ne \i$, the product in \eqref{l ne 1.1} contains the term with $j=l$, and by \eqref{l ne 1}, it is this term that decays as $O(k^{-\frac2d}).$
So, those eigenvalues of $\BF_\i$, that have the form $\pb_\i(\bl^{\pm,l}_k)$  with $l\ne\i$, decay as $O(k^{-\frac2d})$, i.e.,  faster than $k^{-\frac1d}$, and therefore they  make zero contributions to the right-hand  side of \eqref{B_1}. As for the eigenvalues of the operator $\AF$ lying \emph{in} the neighborhood of $\o_\i$, we have
\begin{equation}\label{l=1}
    \pb_\i(\bl^{\pm,\i}_k)=(\bl^{\pm,\i}_k-\o_\i)\prod_{l\ne\i}(\bl^{\pm,\i}_k-\o_l)^2\sim(\bl^{\pm,\i}_k-\o_\i)\prod_{l\ne\i}(\o_\i-\o_l)^2,\, k\to\infty
\end{equation}
as $k\to\infty$, and this justifies the existence of the limit and the equality in \eqref{B_1}
\end{proof}

 Now we calculate the principal and subprincipal symbols  for the operator $\BF_\i$. Here and further on, the superscript '\dag' in  sums and products denotes that the terms where the counting parameter equals $\i$ are omitted.
 In the process of calculating  the symbol of $\pb_\i(\AF)$, to be presented now, we  fix some local co-ordinates and frames near the point $x$ and will trace the  leading two terms in the classical expansion of the symbol of pseudodifferential operators involved. Thus, the relation $\BF\sim \bF_0(x,\x)+\bF_{-1}(x,\x)$ denotes that $\bF_0$ and $\bF_{-1}$ are the  leading two terms, of order $0$ and $-1$, of the symbol of the operator $\BF$ in the chosen (and fixed) local co-ordinate system and the frame in the fiber of the bundle $\Ec$.

So we have
\begin{gather}\label{sym3}
    \AF\sim \aF_0+\aF_{-1};\,
    \AF-\o_l\sim (\aF_0-\o_l)+\aF_{-1};\\ \nonumber
    (\AF-\o_l)^2\sim (\aF_0-\o_l)^2 +(\frac1{i}\partial_\x \aF_0\partial_x \aF_0+(\aF_0-\o_l)\aF_{-1}+\aF_{-1}(\aF_0-\o_l)),
\end{gather}
by the product rule. As it is usual in this kind of calculations, in our notation,  products involving the derivatives $\partial_{\x} $ and $\partial_x $, like the ones in \eqref{sym3}, are understood as sums over $\a$ of such
products containing the partial derivatives $\partial_{\x_\a} $ and $\partial_{x_\a}$, for example,
 \begin{equation*}
    \partial_{\x} \aF_0(x,\x)  \bF(x,\x) \partial_{x} \aF_0(x,\x):= \sum_\a \partial_{\x_\a} \aF_0(x,\x)  \bF(x,\x) \partial_{x_\a}\aF_0(x,\x).
\end{equation*}
Also, for the sake of convenience, we denote $\frac1{i}\partial_\x \aF_0\partial_x \aF_0$ by $\uF(x,\x)$,
 $\partial_\x \aF_0(\aF_0-\o_j)+(\aF_0-\o_j)\partial_\x \aF_0$ by $\vF_j(x,\x)$, and $\partial_x \aF_0(\aF_0-\o_j)+(\aF_0-\o_j)\partial_x \aF_0$ by $\wF_{j}(x,\x)$.
 In this notation, for the symbol of $\AF_{(\i)}=\prod^{\dag} (\AF-\o_l)^2$, we have
\begin{gather}\label{sym4}
  \AF_{(\i)}\sim {\prod_{l}}^{\dag}(\aF_0-\o_l)^2 (\mathrm{\, and\, this\, is\, the\ leading\, term  }:= \yF_0(x,\x))\\ \nonumber +
    { \sum_{1\le j<l\le \Lb}}^\dag\left(\frac1{i}{\prod_{ s<j}}^\dag(\aF_0-\o_s)^2 \vF_j(x,\x)
    {\prod_{j<s<l}}^{\dag}(\aF_0-\o_s)^2\wF_l(x,\x){\prod_{s>l}}^{\dag}(\aF_0-\o_s)^2\right)+
    \\\nonumber \frac{1}{i}\sum_{ j}^{\dag}\left[{\prod_{l<j}}{ }^{\dag}(\aF_0-\o_l)^2\right] \uF(x,\x) \left[{\prod_{ l>j}}^{\dag}(\aF_0-\o_l)^2\right]+
     \\ \nonumber {\sum_{j}}^{\dag} {\prod_{ l<j}}^{\dag}(\aF_0-\o_l)^2((\aF_0-\o_j)\aF_{-1}+\aF_{-1}(\aF_0-\o_j)){\prod_{ l>j}}^{\dag}(\aF_0-\o_l)^2\equiv\\\nonumber
     \yF_0(x,\x)+\yF_{-1}(x,\x).
\end{gather}

Finally, we find the symbol of $\BF_{\i}=\pb_\i(\AF)=(\AF-\o_\i)\AF_{(\i)}:$
\begin{gather}\label{symbol p1(A)}
\BF_{\i}\sim (\aF_0-\o_\i) \yF_0 (\mathrm{\, it\, is\, the\, principal\, symbol\, of\,} \BF_{\i})+\\\nonumber
(\aF_0-\o_\i)\yF_{-1}+\aF_{-1}\yF_0 +\frac1i \partial_\x \aF_0\partial_x \yF_0,
\end{gather}
where $\yF_0$, $\yF_{-1}$ are calculated in accordance with \eqref{sym4}. So, the principal symbol of $\pb_\i(\AF)$ \emph{seems}
to be equal to $(\aF_0(x,\x)-\o_\i)\prod_{l}^{\dag} (\aF_0(x,\x)-\o_l)^2=\pb_\i(\aF_0(x,\x))$. However, the latter expression equals $0$ since it contains the factor $\pb(\aF_0(x,\x))=0$. Therefore,
 the \emph{actual} principal symbol of $\BF_{\i}$ i.e., its symbol of  order $-1$, is
\begin{equation}\label{symbolFin}
    \bF_{\i,-1}(x,\x):= (\aF_0-\o_\i)\yF_{-1}+\aF_{-1}\yF_0 +\frac1i \partial_\x \aF_0\partial_x \yF_0,
\end{equation}
 (unless it vanishes everywhere -- the case to be considered later on).

Together with Lemma \ref{polyn.lem}, we arrive at our main eigenvalue asymptotics  Theorem  for the nondegenerate case, which we formulate now in more detail.

\begin{theorem}\label{main} Let $\AF$ be a zero order polynomially compact self-adjoint pseudodifferential operator on $\G$.
 Then for the eigenvalues of $\AF$ approaching the point $\o_\i$, $\i=1,\dots, \Lb,$ of the essential spectrum, the asymptotic formula holds
\begin{gather}\label{Formula1}
    n_{\pm}(\o_\i, \t)\sim C^{\pm}_{-1}(\o_\i) |\t-\o_\i|^{-d}, \t\to \o_\i\pm 0\, ,\\ \label{Formula2}
    C^{\pm}_{-1}(\o_\i)=d^{-1}(2\pi)^{-d}\int_{S^*\G}\Tr[(\bF_{\i,-1}(x,\x)_{\pm})^{d}]\pmb{\pmb{\o}} dx.
\end{gather}
\end{theorem}

\begin{proof} The statement follows directly from the classical asymptotic formula for eigenvalues of a negative order pseudodifferential
 operator and the explicit expression \eqref{symbolFin} for the symbol of $\BF_{\i}$.  This formula
  was established by M. Birman and M. Solomyak in \cite{BS}. A  detailed exposition for the case of an operator on manifolds was presented
   by G. Grubb in \cite{Grubb84}, see Lemma 4.5 there. We just recall that $dx$ in \eqref{Formula2} is the volume element on $\G$ with respect to
    the fixed Riemannian metric, $S^*\G$ is the cospheric bundle with respect to this metric and $\pmb{\pmb{\o}}$ is the corresponding
     volume form on the cotangent sphere, expressed in local co-ordinates as
\begin{equation*}
\pmb{\pmb{\o}}=\x_1d\x_2\wedge\dots\wedge d\x_d+\dots (-1)^d \x_d d\x_1\wedge\dots d\x_{d-1}
\end{equation*}
(this convenient  expression for the form $\pmb{\pmb{\o}}$ was proposed by L.H\"ormander in \cite{Hormander}.)
\end{proof}
We will refer to the case above, namely, when at least one of the coefficients $C^{\pm}_{-1}(\o_\i)$ does not vanish for the given point $\o_\i$, as the nondegenerate one.

\section{The degenerate case}\label{degen}
In this way, the construction described above establishes the asymptotics of order $|\t-\o_\i|^{-d}$ for the counting function of the
 eigenvalues of $\AF$ near the point $\o_\i$. It may happen, however, that for a certain $\i$ the symbol $\bF_{\i, -1}(x,\x)$ in \eqref{symbolFin}
 is zero for all $(x,\x)\in S^{*}\G$,
  therefore both coefficients $C^{\pm}_{-1}(\o_\i)$ vanish. In this case, the asymptotic formulas \eqref{Formula1}, \eqref{Formula2} are,
  of course, still correct, however, they  are less informative, saying only that
   \begin{equation}\label{oFormula} n_{\pm}(\AF;\o_\i, \t)=o(|\t-\o_\i|^{-d}), \, \t\to\o_\i. \end{equation}
  \begin{remark}\label{rem_degen} We will see in the end that in our motivating example, for the NP operator in elasticity, we encounter \emph{always} the nondegenerate case. However, we present here the way to treat the degenerate case as well, for the sake of completeness.
  \end{remark}
   We show now that our approach  still enables  one to find the power asymptotics of $n_{\pm}(\AF;\o_\i, \t)$,
    as long as it exists, provided, of course, sufficiently many calculations are made. So, we suppose that $\bF_{\i,-1}(x,\x)\equiv 0,$
      $(x,\x)\in \dot{T}^*X$. This means that the operator $\BF_{i, -1}=\pb_\i(A)=(\AF-\o_\i){\prod_{l}}^{\dag}(\AF-\o_l)^2$ is now
       a pseudodifferential operator of order $-2$ and instead of \eqref{B_1} in  the analogy of Lemma \ref{polyn.lem}
        we should expect
         \begin{equation}\label{B_2} \lim_{\t\to +0} \t^{d/2}n_{\pm}(\AF; \o_\i, \o_\i\pm \t)= C \lim_{\t\to+0}\t^{d/2} n_{\pm}(\BF_{\i, -1};\t);
\end{equation}
 However, now the statement in such Lemma,  '\emph{the eigenvalues of $\BF_{\i,-1}$ that have the form $\pb_\i(\s^{\pm,l}_k)$  with $l\ne\i$
  make zero contributions to the right-hand  side of \eqref{B_2}}' is not justified by the estimate \eqref{l ne 1.1}.
  To save the game we need to replace the   operator $B_{\i, -1}$ by another one, namely
  \begin{equation}\label{B1.2}
    \BF_{\i,-2}=\pb_{\i,-2}(\AF)=(\AF-\o_i)\prod{ }^\dag (\AF-\o_l)^3=(\pb(\AF))^3(\AF-\o_\i)^{-2}.
  \end{equation}
  With this  definition, we can prove now the analogy of Lemma \ref{polyn.lem}
\begin{lemma}\label{polyn2.lem} Suppose that for a certain $\i$, the symbol $\bF_{\i,-1}(x,\x)$ equals zero everywhere on $\dot{T}^*\G$. Then the limit
\begin{equation*}
 \lim_{\t\to +0} \t^{d/2}n_{\pm}(\AF; \o_\i, \o_\i\pm \t)
\end{equation*}
exists and it is equal to
\begin{equation}\label{polyn2.degen}
  \lim_{\t\to +0} \t^{d/2}n_{\pm}(\AF; \o_\i, \o_\i\pm \t)  =  \prod{ }^\dag |\o_l-\o_\i|^{3d} \lim_{\t\to+0}\t^{d/2} n_{\pm}(\BF_{\i,-2};\t)
\end{equation}
\end{lemma}
\begin{proof}
With the above definition adopted, the proof of Lemma \ref{polyn2.lem} follows the reasoning in Lemma \ref{polyn.lem} in the following way. We  notice first
that the operator $\BF_{\i,-2}$ equals $\BF_{\i,-1}{\prod_l}^{\dag}(\AF-\o_l)$. By the condition of Lemma, $\BF_{\i,-1}$ is a pseudodifferential operator of order $-2$,
 therefore $\BF_{\i,-2}$ is an operator of order $-2$ as well. Next, by  estimate \eqref{l ne 1}, the eigenvalues of $\AF$ in the
  neighborhood of $\o_l, l\ne \i$, do not contribute to the right-hand side in \eqref{polyn2.degen} (this is why we introduced
   the exponent 3 in the definition of the operator $\BF_{\i,-2}$). Finally, similarly to Lemma \ref{polyn.lem}, those  eigenvalues
    of $\BF_{\i,-2}$ that have the  form $\pb_{\i,-2}(\bl)$ for $\bl$ being eigenvalues of $\AF$ near $\o_\i$, have the asymptotics
     determined by the eigenvalue asymptotics of the latter ones, similarly to \eqref{l=1}.
\end{proof}
Having Lemma \ref{polyn2.lem} at disposal, we can find the asymptotics of the eigenvalues of $\AF$ near $\o_\i$, by calculating the terms in
the  symbol of the operator $\BF_{\i,-2}$ and applying the result by M.Birman and M.Solomyak, this time of the order $-2$ pseudodifferential operator. This can be done along the lines of the calculations in the previous section, however
 we need now to trace three terms in the expansions of the symbols involved. After all calculations are made, it turns out that
 the symbols of order $0$ and $-1$ in $\BF_{\i,-2}$ vanish, and it is the symbol of order $-2$ that is, actually, the principal symbol of $\BF_{\i, -2}$,
   and it is this symbol that determines the eigenvalue asymptotics by the formula similar to \eqref{Formula1}, with
    replacement of the exponent $d$ by $\frac{d}{2}$.

A similar reasoning takes care of the situation when even the principal symbol of $\BF_{\i,-2}$, the one of order $-2$, vanishes as well, i.e., the operator $\BF_{\i,-2}$
 has, in fact,  order $-3$. In this case, we should consider the operator
 $$\BF_{\i,-3}= \pb_{\i,-3}(\AF)=(\AF-\o_\i)\left(\prod{}^{\dag} (\AF-\o_l)^4\right)=(\pb(\AF))^4(\AF-\o_\i)^{-3}, $$
 with the resulting eigenvalue asymptotics having order $\frac{d}{3}$,
 as so on. Of course,  formulas for the principal symbol of $\BF_{\i,-\mathbf{l}}$ become  more and more complicated as long as $\mathbf{l}$, the order of degeneracy,  grows,
  and it will involve more and more homogeneous terms in the symbol of the operator $\AF$.

  We can now summarize the reasoning above in a condensed way.
  \begin{theorem}\label{mostmain}Let $\AF$ be a polynomially compact self-adjoint classical zero order pseudodifferential operator on a smooth closed
    manifold $\G$ and let $\o_\i$ be a point in the essential spectrum. Then the following options are possible:
    \begin{enumerate}
    \item For some $\lb\in\N_+$, the eigenvalues of $\AF$ approaching  $\o_\i$ have asymptotics
    \begin{equation}\label{final_asymp}
        n_{\pm}(\AF, \o_\i, \o_\i\pm \t)\t^{d/\lb}\to C^{\pm}_{\lb}(\o_\i),\, \t\to +0,
    \end{equation}
    with at least one of the coefficients $C^{\pm}_{\lb}(\o_\i)$ nonvanishing;  in this case the corresponding  coefficient can be expressed via $\lb+1$
    highest terms in the symbol of $\AF$;
    \item for any $\d>0$,
    \begin{equation}\label{final_asympt_DEG}
       n_{\pm}(\AF, \o_\i, \o_\i\pm \t)=o(\t^{-\d}), \, \t\to +0,
    \end{equation}
    and we have infinite degeneracy.
    \end{enumerate}
 \end{theorem}

 \section{Symmetrizable operators}\label{Sect.Sym}
 In the application below, we encounter polynomially compact zero order  pseudodifferential operators which are not self-adjoint in $L^2(\G)$ but are self-adjoint as considered in some other Hilbert
 space, actually, the Sobolev space $H^{-\frac12}(\G)$, with norm defined in a special way. The results of the above sections cannot be applied here automatically
since the Theorem by M. Birman and M. Solomyak
on the asymptotics of spectrum of negative order pseudodifferential operators, which we used, is established formally for the $L^2$ space only.

So, let $\KF$ be a polynomially compact zero order pseudodifferential operator in $L^2(\G)$ with polynomial $\pb(\o)$,  and suppose that it is symmetrizable in the following sense. There exists an elliptic  order $-2s$, $s\in\R\setminus \{0\},$ pseudodifferential operator $\SF$, positive in $L^2$, such that   $\SF^{\frac12}: L^2(\G)\to H^{s}(\G)$ is an isomorphism and the operator $\AF=\SF^{\frac12}\KF\SF^{-\frac12}$ is self-adjoint in $L^2(\G)$. We will call such operators $\KF$ \emph{symmetrizable}. The operator $\AF$ is polynomially compact as well, with the same polynomial $\pb(\o)$, and therefore with the same essential spectrum as $\KF$. The eigenvalues of $\KF$ and $\AF$ with eigenfunctions in $C^{\infty}(\G)$, obviously coincide. We will need some more, namely that the $L^2$--eigenvalues of these two operators coincide as well.
To prove it, we need a  general statement about the smoothness of eigenfunctions. For pseudodifferential operators of positive order
this property  is a trivial
consequence of ellipticity. For operators of order zero, a somewhat different (almost very simple)  reasoning is needed.

\begin{proposition}\label{invariance}Let $\KF$ be a polynomially compact symmetrizable zero order pseudodifferential operator, with $\Ob$ being the set of points
 of its essential spectrum. Suppose that for some $\b\in\R$, a distribution $u\in H^{\b}(\G)$ satisfies $\KF u=\bl u $ with some  $\bl\not\in\Ob$. Then  $u\in C^{\infty}(\G)$.
\end{proposition}
 \begin{proof} Consider the operator $\MF=\pb(\KF)$.  We have $\MF u= \pb(\bl)u,$ with $\pb(\bl)\ne0$. Therefore,
 $u=\pb(\bl)^{-1}\MF u$. We iterate to obtain $u=(\pb(\bl))^{-n}\MF^{n}u$. Since $\MF$ is a   pseudodifferential operator of order $-1$, $\MF^{n}$ is an operator of order $-n$,
 $\MF^{n}:H^{\b}\to H^{\b+n}$, therefore $u\in H^{\b+n}(\G)$ for any $n$.
 \end{proof}

 Now we can show that the sets of eigenvalues outside $\Ob$ of $\pb(\AF)$ and of $\pb(\KF)$ in $L^2$ coincide. Let $v\in L^2$ be an eigenfunction of the operator $\AF$ corresponding to an eigenvalue $\bl$ outside $\Ob$. By Proposition, $v\in C^{\infty}$. Then $\SF^{-\frac12}v\in C^{\infty}$ is  an eigenfunction of $\KF$ with the same eigenvalue. At the same time, if $u\in L^2$ is an eigenfunction of $\KF$, it must belong to $C^{\infty}$, therefore the function $v=\SF^{\frac12}u\in C^{\infty}$ is an eigenfunction of $\AF$.
Next, let $\MF_{\i,-1}=\pb_{\i}(\KF)$ be the order $-1$ pseudodifferential operator constructed as in \eqref{polynB},
with principal symbol $\mF_{\i,-1}$. Since $\pb_\i{(\AF)}=\SF^{\frac12}\pb_\i(\KF)\SF^{-\frac12}$, the eigenvalues of $\pb(\KF)$ coincide with eigenvalues of $\pb({\AF})$. Also, the principal, order $-1$, symbols of these operators satisfy
\begin{equation}\label{SimilarSymb}
    \mF_{\i,-1}(x,\x)=\ssF(x,\x)^{-\frac12}\bF_{\i,-1}(x,\x)\ssF(x,\x)^{\frac12},
\end{equation}
where $\ssF(x,\x)$ is the principal symbol of $\SF$.  The symbol $\mF_{\i,-1}$ is not necessarily a Hermitian matrix, but being similar to the Hermitian one, to $\bF_{\i,-1}$, it has real eigenvalues, the same as $\bF_{\i,-1}$. We denote by $\Tr((\mF_{\i,-1}(x,\x)_\pm)^d)$ the quantity $\Tr((\aF(x,\x)_\pm)^d)$; due to the similarity just mentioned, it equals the sum of $d$-powers of positive, resp., of absolute values of negative, eigenvalues of the matrix $\mF_{\i,-1}(x,\x)$.

This enables us to justify  the eigenvalue asymptotics for the operator $\KF$.
\begin{theorem}\label{SymmAS} Let $\KF$ be a polynomially compact zero order symmetrizable pseudodifferential operator, as above.
Then for the eigenvalues of $\KF$ approaching the point $\o_{\i}$, the formula
\begin{gather}\label{Formula3}
    n_{\pm}(\KF;\o_\i, \t)\sim C^{\pm}_{-1}(\o_\i) |\t-\o_\i|^{-d}, \t\to \o_\i\pm 0\, ,\\ \label{Formula4}
    C^{\pm}_{-1}(\o_\i)=d^{-1}(2\pi)^{-d}\int_{S^*\G}\Tr[(\mF_{\i,-1}(x,\x)_{\pm})^{d}]\pmb{\pmb{\o}} dx,
\end{gather}
similar to \eqref{Formula1}, \eqref{Formula2}, holds, with the symbol $\mF_{\i,-1}(x,\x)$ obtained by the formulas \eqref{sym4}, \eqref{symbol p1(A)}, \eqref{symbolFin}, just with $\aF_{0},\aF_{-1}$ replaced by $\kF_0,\kF_{-1}$.
\end{theorem}

It is important to explain here the dependence of the expression
\begin{equation}\label{integrand}
   \pmb{\pmb{\g}}(x) =\int_{S^*(\G)_x}\Tr[(\mF_{\i,-1}(x,\x)_{\pm})^{d}]\pmb{\pmb{\o}}
\end{equation}
 on the choice of the frame in the fiber $\Ec_{x}$ and on the choice of local co-ordinates on $\G$. First, under the change of the frame in the fiber, the matrix representing the symbol $\mF_{\i,-1}(x,\x)$ is changed to a similar matrix, the one with the same eigenvalues, therefore this change does not affect the integrand in \eqref{integrand}.

We consider now the change of local co-ordinates, $x\to \tilde{x}=Z(x)$, at a neighborhood of a point $x^{\circ}\in\R^d$, so that the Jacobi matrix $U$ for this mapping is orthogonal at $x^\circ$. Then, by the usual rule of transformation of the symbol under the co-ordinates change, the symbol $\mF_{\i,-1}(x,\x)$, calculated at the point $x^{\circ}$, transforms to $\mF_{\i,-1} (x,U \x)$, and thus for the integral in $x=x^{\circ}$ this change of variables reduces to an orthogonal transformation of $\x$ in the integrand. Therefore, the expression \eqref{integrand} does not change either.


 \section{The Neumann-Poincare operator in 3D elasticity.}\label{NP}
 To demonstrate how the above approach works in a concrete setting, we discuss our motivating example. Let $\Db\subset\R^3$ be an elastic bounded body
 with smooth ($C^{\infty}$) boundary $\G$. The structural properties of the body $\Db$ are supposed to be homogeneous and isotropic.
 The linear second order system describing the deformation of $\Db$, the Lam\'e - Navier
 system has the form
 \begin{gather}\label{Lame}
     \Lc \ub\equiv \Lc_{\m,\l}\ub\equiv\div(\m \grad \ub)+\grad((\l+\m)\div \ub)=0,\\  \nonumber
      \xb=(x_1,x_2,x_3)\in\Db, \ub=(u_1,u_2,u_3)^\top,
 \end{gather}
 where $\l,\m$  are the  Lam\'e constants.

 The fundamental solution  $\Rc(\xb,\yb)=[\Rc(\xb,\yb)]_{p,q=1,2,3}$ for the Lam\'e equations, \emph{the Kelvin matrix},  known since long ago, see, e.g., \cite{KuprPot},
 equals
 \begin{gather}\label{Kelvin}
    [\Rc(\xb,\yb)]_{p,q}=\l'\frac{\d_{p,q}}{|\xb-\yb|}+\m'\frac{(x_p-y_p)( x_q-y_q)}{|\xb-\yb|^3},\\ \nonumber \l'=\frac{\l+3\m}{4\pi \m(\l+2\m)},
    \, \m'=\frac{\l+\m}{4\pi\m(\l+2\m)}, \, p,q=1,2,3.
 \end{gather}
 This expression can be found, in particular, by inverting the Fourier transform of the symbol $\rb(\bx)$of $\Lc^{-1}$:
 \begin{equation}\label{KelvinF}
    \Rc(\xb,\yb)=\Fc^{-1}[\rb](\xb-\yb)\equiv(2\pi)^{-3}\int_{\R^3} e^{i(\xb-\yb)\bx} (\m \bx\bx^\top+(\l+\m)|\bx|^2\Eb)^{-1} d\bx,
 \end{equation}
where $\bx$ is treated as a column-vector, so that $\bx\bx^\top$ is a  $3\times 3$ square matrix; $\Eb$ is the unit $3\times 3$ matrix.

 For this equation, several types of boundary problems are usually considered, see, e.g., \cite{Kupr79}, Chapter 1, $\S$12--14.
  They involve the stress  operator $T$. This coboundary operator has the matrix representation
  \begin{equation}\label{traction}
    [T(\xb,\pn)]_{p,q}=\l \n_p\pd_q+\m\n_q\pd_p+\m\d_{p,q}\pnu,
  \end{equation}
  where $\nup=\nup(\xb)=(\n_1,\n_2,\n_3)$ is the unit outward normal vector to $\G$ at the point $\xb$ and $\pnu_{(\xb)}$ is the
   directional derivative along $\nup(\xb)$.
 The description of four basic boundary problems for the system \eqref{Lame} is given, e.g.,  in \cite{Kupr79}, a
  discussion about their pseudodifferential representation can be found in \cite{Kozh1}, \cite{Kozh2}.
   The analysis of these problems uses  proper versions of the double layer potentials for  the equation \eqref{Lame}; to the
   'second basic problem' there corresponds the potential
 \begin{equation}\label{DoubleLayer}
    (\Dc[\psi])(\xb)=\int_{\G}\DF(\xb,\yb)\psi(\yb)dS(\yb)\equiv \int_{\G} T(\yb,\pnu_{(\yb)}) \Rc(\xb,\yb)^{\top}\psi(\yb)dS(\yb), \xb\in \Db,
 \end{equation}
 where $dS$ is the natural surface  measure  on $\G$ and $T(\yb,\pnu_{(\yb)})$ denotes the coboundary operator at the point $\yb\in\G.$

 For $\psi\in L^2(\G)$ (and even $\psi\in \mathscr{D}'(\G)$), the  integral in \eqref{DoubleLayer}
  is well defined for $\xb\in\Db$. When the point $\xb$ lies on the boundary $\G$, the kernel in the integral
 has singularity of order $|\xb-\yb|^{-2}$, and therefore it is a singular integral operator in $L^2(\G)$, denoted by $\KF$.
  It is called the Neumann-Poincar\'e
  operator for 3D elastostatics.

  The explicit representation for the integral kernel of $\KF$ has been found in \cite{Kupr79}, Ch.2, \S 4:
\begin{gather}\label{Kernel}
    [\Kc(\xb,\yb)]_{p,q}=\m(\l'-\m')\frac{\n_p(\yb)(x_q-y_q)-\n_q(\yb)(x_p-y_p)}{|\xb-\yb|^3}+\\ \nonumber
    \left( \m(\m'-\l')\d_{p,q}-6\m\m'\frac{(x_p-y_p)(x_q-y_q)}{|\xb-\yb|^2}\right)\sum_{l=1}^3 \n_l(\yb)\frac{x_l-y_l}{|\xb-\yb|^3}; \xb, \yb \in\G.
\end{gather}
The diagonal ($p=q$) terms in the kernel \eqref{Kernel} have a weak, $O(|\xb-\yb|^{-1})$ singularity as $|\xb-\yb|\to 0$, so they  define weakly polar, and, therefore, compact,  integral operators. The off-diagonal ($p\ne q$) terms have singularity of order $|\xb-\yb|^{-2}$, odd in $\xb-\yb$, therefore they define bounded singular integral operators.

It is convenient to use a special co-ordinate system to study the integral operator $\KF$. Following \cite{AgrLame},
 for a fixed point $\xb^\circ\in\G$, we place the origin of the co-ordinate system at $\xb^\circ$, and fix
  some orthogonal co-ordinates $x =(x_1,x_2)$ on the tangent plane at $\xb^\circ$ (naturally identified with the cotangent one), directing $x_3$ axis along the exterior normal at $\xb^\circ$.
 In these co-ordinates, the surface near $\xb^\circ$ is defined by
 the equation $x_3=F(x)\equiv F(x_1,x_2)$, $F(0)=0, \nabla F(0)=0$.
    The dual co-ordinates $\x=(\x_1,\x_2)$ in the cotangent bundle are chosen correspondingly to ${x}$.
The orthogonal frame in which the action of the operator will be represented is adapted to the chosen
 co-ordinate system, i.e., two vectors in the frame are directed along
the $x_1,x_2$ axes in the tangent plane at $\xb^\circ$, with the third  vector directed outward along the normal vector $\nup(\xb)$. Later on, we will specify in more detail the choice of co-ordinates $\xb$.

Due to the assumed smoothness of the surface $\G$, all entries in the expression of the kernel \eqref{Kernel} can be
 expanded in asymptotic series at the point ${y}={x}$, with terms, positively homogeneous in ${x}-{y}$ (this concerns also the Jacobian in the integral arising under the change of co-ordinates), the leading term being of order $-2$. This fact  establishes in the usual way that  the operator $\KF$ is a classical zero-order pseudodifferential operator on the surface $\G$.

  In \cite{AgrLame}
the principal symbol of the pseudodifferential operator  $\KF$ was calculated:
 \begin{equation}\label{NPPrinc}
    \kF_0({x},\x)=\frac{i\pi \m(\l'-\m')}{|\x|}\begin{pmatrix}
                    0 & 0 & -\x_1 \\
                    0 & 0 & -\x_2 \\
                    \x_1 & \x_2 & 0 \\
                  \end{pmatrix}
 \end{equation}
 It follows that the principal symbol \eqref{NPPrinc} has eigenvalues $0, \pm \mathbbm{k} $, where $\mathbbm{k}=\pi\m(\l'-\m')=\frac{\m}{2(2\m+\l)}$,
  so they are independent of $(x,\x)\in T^*(\G)$ and therefore the operator $\KF$ is polynomially compact in $L^2(\G)$, with polynomial
   $\pb(\o)=\o(\o^2-\mathbbm{k}^2)$. It has the essential spectrum consisting of the points $0, \pm \mathbbm{k}$.

   This operator is not self-adjoint in $L^2(\G)$, this is easily visible from its definition \eqref{DoubleLayer} -- the adjoint operator involves the normal derivative at the point $\xb$, and not at the point $\yb$, as in \eqref{DoubleLayer}. This shortcoming can be circumvented by showing that $\KF$ is symmetrizable.

    Consider the \emph{single layer operator} on $\G$:
    \begin{equation}\label{SingleLayer}
        \SF[\psi](\xb)=\int_\G \Rc(\xb,\yb)\psi(\yb) dS(\yb),\, \xb\in \G,
    \end{equation}
    the kernel $\Rc$ being defined in \eqref{KelvinF}.
    This is a self-adjoint operator in $L^2(\G)$. It is well known, see, e.g.,  \cite{AgrLame}, that $\SF$
    (it is denoted by $A$ there) is an elliptic
    pseudodifferential operator of order $-1$. Therefore, $\SF$ maps the Sobolev space $H^s(\G)$
     into the space $H^{s+1}(\G)$ for any $s\in(-\infty,\infty)$.
     The principal symbol of $\SF$ has been calculated in in  \cite{AgrLame}, Sect. 1.6.
      In the local co-ordinates and the frame just used above, it has the block-matrix form
      \begin{equation}\label{SingleSymbol}
        \ssF_{-1}(x,\x)=\frac{1}{2\m|\x|}\left(\frac{\l+\m}{2(\l+2\m)}\begin{pmatrix}
                                             \L(\x) & 0 \\
                                             0 & 1 \\
                                           \end{pmatrix}-\Eb\right).
      \end{equation}
Here $\L(\x)$ is the matrix
\begin{equation}\label{Lambda}
    \L(\x)=|\x|^{-2}\begin{pmatrix}
         \x_1^2 & \x_1\x_2 \\
        \x_1\x_2& \x_2^2 \\
       \end{pmatrix},
\end{equation}
$\Eb$ denotes the unit $3\times 3$ matrix.

The matrix \eqref{SingleSymbol} is invertible, therefore, the operator $\SF$ is elliptic. We need some more, namely, that $-\SF$ is positive in $L^2(\G)$.

\begin{proposition}\label{Prop.Positive} The single layer potential $\SF$ is a negative operator in $L^2(\G)$, $\la\SF \psi,\psi\ra_{L^2(\G)}<0.$
 \end{proposition}
 \begin{proof}
 In the scalar case, for the single layer electrostatic potential, this property is well-known,
 see, e.g.,  \cite{Landkof}, Theorem 1.15. We could not find the reasoning for the elastic case in the literature, therefore we present an elementary proof here.
 Denote by $\Qc(\xb,\yb)$ the fundamental solution of the square root of the minus Lam\'e operator $-\Lc$. This function can be constructed as
 \begin{equation*}
    \Qc(\xb,\yb)=(2\pi)^{-3}\int_{\R^3} e^{i(x-y)\bx}\sqrt{-\rb(\bx)}d\bx,
 \end{equation*}
 with $\sqrt{\cdot}$ denoting here the positive square root of a positive matrix. Since $\sqrt{-\rb(\bx)}\times\sqrt{-\rb(\bx)}=-\rb(\bx)$, the kernel $\Qc$ satisfies
 \begin{equation}\label{double}
    \int_{\R^3} \Qc(\xb,\zb)\Qc(\zb,\yb) d\zb =-\Rc(\xb,\yb), \xb\in \Db.
 \end{equation}
 Using \eqref{double}, we can represent the single layer operator as $\SF=-\Qb^*\Qb$, where $\Qb$ is the operator acting from $L^2(\G)$ to $L^2(\R^3)$ as
 $\Qb[\psi](\xb)=\int_\G \Qc(\xb,\yb)\psi(\yb)dS(\yb)$. This representation shows that the operator $-\SF$ is nonnegative. Finally, in accordance with  \cite{AgrLame}, Proposition 1.2, the null space of $\SF$ is trivial, so $-\SF$ is positive.
 \end{proof}

 Taking into account the ellipticity of $\SF$, we know now that $-\SF$ is an isomorphism of Sobolev spaces, $-\SF: H^s(\G)\to H^{s+1}(\G), \, -\infty<s<\infty$. Moreover, any power of $-\SF$ is an isomorphism $(-\SF)^\a: H^s(\G)\to H^{s+\a}(\G)$, $-\infty<\a<\infty.$

Matrix $\L=\L(\x')$ satisfies $\L^2=\L$, this property enables us to calculate symbols of some operators related with $\SF$. First, the inverse $\RF=\SF^{-1}$ is a pseudodifferential operator of order 1. Its principal symbol $\rF_1$ equals $\ssF^{-1}$,
 \begin{equation}\label{rF}
    \rF_1(\x)=2\m|\x|\left(-\frac{\l+\m}{\l+3\m}\begin{pmatrix}
                    \L(\x) & 0 \\
                    0 & 1 \\
                  \end{pmatrix}-\Eb\right).
 \end{equation}

  We will also need  the  (positive) square roots of the operators $-\SF$ and $-\RF$. The operator $\QF=(-\SF)^{\frac12}$ is an elliptic pseudodifferential operator
 of order $-\frac12$ and its principal symbol equals
\begin{equation}\label{qF}
    \qF_{-\frac12}(\x')=(-\ssF_{-1}(\x'))^{\frac12}=\frac{1}{(2\m|\x|)^{\frac12}}\left(\Eb-(1-(1-m)^{\frac12})
    \begin{pmatrix}
    \L & 0 \\
     0 & 1 \\
     \end{pmatrix}\right),
     \end{equation}
where $m=\frac{\l+\m}{2(\l+2\m)}.$ In its turn, the principal symbol of the  order $\frac12$ pseudodifferential operator $\ZF=\QF^{-1}=(-\RF)^{\frac12}$ equals
\begin{gather}\label{zF}
    \zF_{\frac12}(\x)=(-\rF_1(\x))^{\frac12}=(\qF_{-\frac12}(\x))^{-1}=\\\nonumber
    (2\m|x|)^{\frac12}\left(\Eb+ (\frac{1}{\sqrt{1-m}}-1)\begin{pmatrix}
                    \L(\x) & 0 \\
                    0 & 1 \\
                  \end{pmatrix}\right).
\end{gather}
Now we can show that  the operator $\KF$ is symmetrizable in $L^2(\G)$. The following version of the
 \emph{Plemelj formula}
is valid (\cite{Duduchava}, p.89, see also \cite{AgrLame}, Proposition 1.8):
\begin{equation}\label{Plemelj}
    \KF \SF=\SF \KF',
\end{equation}
where $\KF'$ is the  $L_2(\G)$-- adjoint of $\KF$. Equality \eqref{Plemelj} can be also written as
\begin{equation}\label{Plemeli2}
    (-\SF)^{-\frac12}\KF(-\SF)^{\frac12} =(-\SF)^{\frac12} \KF'(-\SF)^{-\frac12}.
\end{equation}
Consequently,  \eqref{Plemeli2}  means that the operator $\KF$ is symmetrizable in $L^2(\G)$: $\AF= (-\SF)^{-\frac12}\KF(-\SF)^{\frac12}=      \QF \KF \ZF$ is
 self-adjoint in $L^2(\G)$.
 It is a zero order classical pseudodifferential operator, with the same spectrum as $\KF$. We will find now its principal symbol. Using our previous calculations, we obtain for the principal symbol of $\AF$,
 \begin{equation}\label{aF}
    \aF_0(\x)=\qF_{-\frac12}(\x)\kF_0(\x)\zF_{\frac12}(\x).
 \end{equation}
Now we notice that the matrices $\kF_0(\x)$ and $\begin{pmatrix}
                    \L(\x) & 0 \\
                    0 & 1 \\
                  \end{pmatrix}$ commute. Taking into account the expression for the principal symbols of $(-\SF)^{\frac12}$ and $(-\SF)^{-\frac12}$,
                  we obtain
\begin{equation}\label{SymbolFin}
    \aF_0(\x)=\kF_0(\x)=\frac{i\pi \m(\l'-\m')}{|\x|}\begin{pmatrix}
                    0 & 0 & -\x_1 \\
                    0 & 0 & -\x_2 \\
                    \x_1 & \x_2 & 0 \\
                  \end{pmatrix}.\end{equation}
Thus, Theorem \ref{mostmain} can be applied to the operator $\AF$ giving the asymptotics of the eigenvalues converging
 to the points of the essential spectrum.


\section{The structure of the symbol $\mF_{\i}$}\label{Sect.structure} By Theorem \ref{SymmAS}, in order to find the coefficient in the asymptotic formula for eigenvalues, we need to calculate the principal symbols $\mF_{\i}\equiv\mF_{\i, -1},$ $\i=1,2,3$ of the $-1$ order pseudodifferential operators $\MF_{\i}=\pb_{\i}(\KF)$.  We aim to avoid the (very tedious) direct calculation of these symbols and their bulky and unwieldy expression.
 We just describe here  the structure of the resulting formula and explain the procedure of obtaining it.

\begin{theorem}\label{Result} For the NP operator  $\KF$, the essential spectrum consists of 3 points $\o_1=0, \o_2=-\mathbbm{k}, \o_3=\mathbbm{k}$. There are infinitely many eigenvalues of $\KF$ near each of these points, and they are nondegenerate in the sense of Sect.~\ref{nondeg}. The coefficients $C^{\pm}_{-1}(\o_\i)$ are obtained in  terms of positive and negative parts of eigenvalues of $3\times 3$ matrices $\mF_{\i}(x,\x)$  depending linearly  on the  principal curvatures  $\kb_1(\xb), \kb_2(\xb)$ of the surface $\G$ at $\xb$, being integrated over the cospheric bundle of the surface $\G$. The matrix $\mF_{\i}(x,\x)$, in proper co-ordinates depending on the geometry of $\G$, has the form
\begin{equation}\label{SymbolStructure}
    \mF_{\i}(x,\x)=\kb_1(x)M^{(1)}_{\i}(\x_1,\x_2)+\kb_2(x)M_{\i}^{(2)}(\x_1,\x_2),
\end{equation}
with  universal, depending only on $\l,\m,$ matrices $M_{\i}^{(1)}(\x_1,\x_2)$, $M_{\i}^{(2)}(\x_1,\x_2)$, order $-1$ positively homogeneous in $\x$, depending on the Lam\'e constants $\l,\m$ but not depending on the surface $\G$.
\end{theorem}

We would like to stress that the representation \eqref{SymbolStructure} is valid only in the specially selected co-ordinates system and frame. These are chosen depending on the geometry of $\G$. However, recall, the eigenvalues of the symbol $\mF_{\i}(x,\x)$ do not depend on the co-ordinates chosen.
\subsection{C--co-ordinate systems}\label{subsect coord}
We choose a special co-ordinate system near a point $\xb^\circ\in\G$, where the structure of the symbol is more visible. We will call it the 'C--co-ordinates at $\xb^\circ$'. It is in this system the representation \eqref{SymbolStructure} is valid.

Suppose that $\xb^\circ$ is an umbilical point of the surface (recall that a point on a surface is called umbilical if the principle curvatures at this point coincide.) For such point, we direct the orthogonal  $x_1,x_2$ axes arbitrarily  in the tangent plane to $\G$ at $\xb^{\circ}$ and the  $x_3$ axis orthogonally to them, in the outward direction. It is known that any smooth compact surface possesses at least two umbilical points.

If $\xb^{\circ}$ is not an umbilical point, we direct $x_1,x_2$ axes along the directions of principal curvatures of $\G$ at $\xb^{\circ}$ and the $x_3$ axis along the outward normal.

In both cases, the surface near $\xb^{\circ}$ is described by the equation $x_3=F(x_1,x_2)\equiv F(x)$ with
\begin{equation}\label{Function F}
    F(0,0)=0;\, \nabla F(0,0)=0, F(x_1,x_2)=\frac12(\kb_1(\xb^{\circ})x_1^2+\kb_2(\xb^{\circ})x_2^2)+O((x_1^2+x_2^2)^{3/2}),
\end{equation}
where $\kb_1(\xb^{\circ}),\kb_2(\xb^{\circ})$ are the principal curvatures of $\G$ at $\xb^{\circ}$.

We will also need certain co-ordinate systems at points $\xb'\in\G$ near $\xb^{\circ}$, consistent with the C--co-ordinates system at $\xb^{\circ}$. For a point $\xb'\in\G$ with co-ordinates $(x_1',x_2',x_3')\equiv(x',x_3'(=F(x')))$, we consider the projection $P_{\xb'}$ of the tangent plane at $\xb^{\circ}$,  $T_{\xb^\circ}(\G)$, to the tangent plane  $T_{\xb'}(\G)$. The co-ordinates $y=(y_1,y_2)$ on $T_{\xb'}(\G)$ will be generated on $T_{\xb'}(\G)$ from $T_{\xb^\circ}(\G)$  by this projection, with $y_3$ axis directed along the normal at $\xb'$ to $\G$. What follows from this construction, is that the Jacobi matrix of  this co-ordinate  transformation is $\Eb+O(|\xb^{\circ}-\xb'|)$, as $\xb'\to\xb^{\circ}$, with derivatives of this Jacobi matrix, due to \eqref{Function F} and the chain rule are, up to $O(|\xb^{\circ}-\xb'|),$ a linear functions of the principal curvatures $\kb_1(\xb^{\circ}),\kb_2(\xb^{\circ})$  at $\xb^{\circ}$. We direct the vectors in the frame of the fibre of the bundle at $\xb'$ along the co-ordinate axes. The transformation matrix to the standard frame at $\xb^{\circ}$ will have the form $\Eb+O(|\xb^{\circ}-\xb'|)$ as well, with derivatives, linear in $\kb_1(\xb^{\circ}),\kb_2(\xb^{\circ})$, up to a higher order term, as $\xb'\to\xb^{\circ}$.

\subsection{The composition of the symbol $\mF_\i$}
We pass to the study of the symbol $\mF_\i(x,\x)$. Recall that it is constructed following the rules \eqref{sym4}, \eqref{symbol p1(A)}, \eqref{symbolFin}, with the symbol of the operator $\AF$ replaced by the symbol of $\KF$.
The expression \eqref{symbolFin} is a sum of many terms. We consider  the structure of these terms; to do this we need certain book-keeping.

Taking into account \eqref{symbolFin}, we see that the expression for the principal
(order $-1$) symbol of the operator $\MF_{\i}$ involves the principal symbol $\kF_0$ of the operator $\KF$,  its first order derivatives in $x$ and $\x$ and, finally, the order $-1$ symbol of $\KF$. We assign the weight $0$ to $\kF_0$ and to the unit matrix, and we assign  the weight $-1$ to $\partial_\x \kF_0\partial_x \kF_0$ and to $\kF_{-1}$, with weights being added when the terms are multiplied. Formulas \eqref{symbolFin}, \eqref{symbol p1(A)}, \eqref{sym4} show that each of the additive terms  in the expression of  the principal symbol $\mF_{\i,-1}$ of $\MF_{\i}$ contains only one factor of weight $-1$, all the rest having weight $0$. We note that we may perform our calculations in any co-ordinate system by our choice. Having this in mind, we will calculate $\mF_{\i,-1}$ in the C--co-ordinate system centered at the point $\xb^{\circ}$.

First, we fix the results of our book-keeping.
\begin{proposition}\label{Structure} The principal symbol  $\mF_{\i,-1}$ of the operator $\MF_{\i}$, expressed in the C--co-ordinates system, can be represented as a sum of several terms, each of them being the product, in some order,  of
\begin{enumerate}\item one  factor  $\kF_{-1}(x,\x)$ and several, no more than $4$, factors $\kF_0$;\\
or
\item one factor  $\partial_{\x}\kF_0(x,\x)$, one factor $\partial_{x}\kF_0(x,\x)$ and several, no more than 3, factors $\kF_0(x,\x)$.
\end{enumerate}
\end{proposition}
In fact, in our case, $\Lb=3$, therefore, the  polynomial $\pb_\i(\o)$ has degree $5$. In the expression for the symbol $\mF_{\i}$, one factor of $\kF_0$ is replaced by $\kF_{-1}$, so no more than 4 factors remain.On the other hand, two factors $\kF_0$ can be replaced   by $\partial_{\x}\kF_0$ and $\partial_{x}\kF_0$, so no more than $3$, factors  $\kF_0$ remain.

The symbol $\kF_0$ does not depend on the geometry of $\G$, as can be seen in \eqref{NPPrinc}. So, the only way how $\mF_{\i}$ can depend on the geometry is via $\nabla_x \kF_0, \nabla_{\x}\kF_0,$ and $\kF_{-1}.$
\subsection{Dependence on the geometry of $\G$. 1. $\nabla_x \kF_0, \nabla_{\x}\kF_0$}
First, we can see from \eqref{NPPrinc}, that the expression for the  symbol $\kF_0$ does not contain any dependence  on $x$, therefore, the same is true for the $\x$ derivatives of $\kF_0$ (these derivatives can be calculated directly from \eqref{NPPrinc}, but we will not do this now). We consider now  $\nabla_{x}\kF_0$.

To find this derivative, we consider the fixed  point $\xb^{\circ}$ and another point $\xb'$ on the surface, in a neighborhood of $\xb^{\circ}$. We consider also the C--co-ordinate system centered at $\xb^{\circ}$ and the consistent system centered at $\xb'$,  as explained in Section \ref{subsect coord}. We will denote by the subscript $\circ$ the principal symbol $\kF_0$ expressed in the $\xb^{\circ}$-- centered system and by $'$ the symbol expressed in the $\xb'$- centered system.

In this notation, we are interested in the derivative $\nabla_{x}\kF_0^{\circ}(x,\x)$ calculated at the point $x=0$, i.e., at $\xb^{\circ}$. We denote by
$Z=Z_{\xb'}$ the variables change  in a  neighborhood of $\xb^\circ$ from $\xb^\circ$-centered co-ordinates to the $\xb'$-centered  ones. The Jacobi matrix $DZ=DZ_{\xb'}$ of this transformation is composed of the first order derivatives of the function $F$  at $\xb'$ The transformation $U(\xb')\in GL(\R,3)$ from the $\xb^{\circ}$-frame to the $\xb'$-frame depends linearly on the first order derivatives of $F$ at  $\xb'$ as well.

 We use now  the classical rule of transformation of the symbol under the change of variables and the natural rule of the change of the basis in the fiber. Namely, we write the symbol in $\xb'$-centered co-ordinates -- it will have the same form as \eqref{NPPrinc} -- and then transform it to $\xb^{\circ}$  co-ordinates.  In this way, we have at the point $\xb'$
\begin{gather}\label{symbK1}
  \kF_0^{\circ}(x,\x)  =  U(\xb')\kF'_0(Z(x),((DZ)^{-1})^\top \x)U(\xb')^{-1}=\\ \nonumber U(\xb')\frac{i\pi \m(\l'-\m')}{|\x|}\begin{pmatrix}
                    0 & 0 & -\y_1 \\
                    0 & 0 & -\y_2 \\
                    \y_1 & \y_2 & 0 \\
                  \end{pmatrix}U(\xb')^{-1},
\end{gather}
 with $\y=((DZ)^{-1})^\top \x.$ We recall here that the transformation $Z(\yb)$, its differential $DZ(\xb')$, and the linear transformations $U(\xb)$ depend smoothly on the first order derivatives of the function $F(x_1,x_2)$, moreover they become identity maps as $\xb'\to\xb^{\circ}$, since the derivatives of $F$ vanish ay $\xb^\circ$. Therefore, the derivatives of $\kF_0^{\circ}(x,\x)$ at $\xb^\circ$, by the chain rule, depend linearly on the second derivatives of $F$ at $\xb^\circ$, with no more characteristics of $F$ involved. Since the co-ordinates $x_1,x_2$ have been chosen along the curvature lines of $\G$ at $\xb^{\circ}$, the mixed second derivative of $F$ vanishes, while the pure second derivatives are equal to the principal curvatures of the surface at the point $\xb^\circ$, moreover, this dependence is linear. Thus, we have established that
 \begin{equation}\label{struct1}
    \nabla_{x}\kF_0^{\circ}(x,\x)=\kb_1(\xb^{\circ})\jF_1(\x)+\kb_2(\xb^{\circ})\jF_2(\x)
 \end{equation}
 in the C--co-ordinate system centered at $\xb^{\circ}$. The same conclusion holds in umbilical points.

\subsection{Dependence on the geometry of $\G$. 2. $\kF_{-1}(x,\x)$}

 In order to find the required representation for the order $-1$ symbol of the operator $\KF$, it is easier to  consider not the symbols but the kernel of  the integral operator.

We consider  the local expression \eqref{Kernel} for the NP operator.  Having the point $\xb=\xb^\circ (=0)$ fixed, we expand all entries of the  kernel $\Kc$ in the asymptotic (Taylor) series in terms, positively homogeneous in $y-x$.  We are interested in the first two terms in this expansion
\begin{gather}\label{KrnelExpnsion}
     [\Kc(\xb,\yb)]_{p,q}=\m(\l'-\m')\frac{\n_p(\yb)(x_q-y_q)-\n_q(y)(x_p-y_p)}{|\xb-\yb|^3}+\\ \nonumber
    \left( \m(\m'-\l')\d_{p,q}-6\m\m'\frac{(x_p-y_p)(x_q-y_q)}{|\xb-\yb|^2}\right)\sum_{l=1}^3 \n_l(\yb)\frac{x_l-y_l}{|\xb-\yb|^3}=\\
   \nonumber \Kc_0(x,x-y)+\Kc_{-1}(x,x-y) +O(1); \, \xb=(x,F(x))\in \G,\,  \yb=(y,F(y))\in\G,
\end{gather}
where $\Kc_0(x,x-y)$ is order $-2$ positively homogeneous and odd in $x-y$ and $\Kc_{-1}(x,x-y)$ is order $-1$ positively homogeneous in $(x-y)$.  In order to achieve it, we consider the expansion for separate terms in \eqref{KrnelExpnsion}. Here we keep in mind the Taylor expansion for the function $F$ near $\xb^{\circ}$,
 $F(x)=\frac{1}{2}(\Hb x,x)+O(|x|^3)$, $x\to 0$, where $\Hb$ is a symmetric $2\times 2$ matrix.  Next, by our choice of co-ordinates,  the co-ordinate axes lie along the eigenvectors of the matrix $\Hb$ (for  an umbilical point, i.e., when $\Hb$ is a multiple of the unit matrix, any orthogonal directions may be chosen.) In this co-ordinate system, the first fundamental form of the surfaces $\G$ at $\xb^\circ$ is the identity one, $\Ib[\G]_{\xb^\circ}(dx)=|dx|^2$. The second fundamental form for this surfaces at $\xb^\circ$ is  diagonal in this co-ordinate system: $\Ib\Ib[\G]_{\xb^\circ}(dx)=\frac12(\kb_1(\xb^\circ)(dx_1)^2+\kb_2(\xb^\circ)(dx_2)^2).$  calculated, recall, for the direction of the normal vector chosen to be the outward one, so it is \emph{negative} at points where the surface is convex.
 For the entries in the kernel for the integral operator $\KF$, at the point  $\xb^{\circ}$ with co-ordinates $(x,F(x))=(0,0)$ in the chosen co-ordinates,
 we use the standard relations for co-ordinates in this system. Namely, for the components of the normal vector, we have  $\n_\a(y)=\n_\a(0)+\kb_\a(\xb^{\circ}) y_\a+ O(|y|^2), \a=1,2,$. For the distance between points,  we have
 \begin{equation*}
 |\xb^\circ-\yb|^2=|x-y|^2(1+2\frac{\Ib\Ib[\G]_{\xb^\circ}(x-y)^2}{|x-y|^2}).
 \end{equation*}

 It follows that two leading terms in the singularity as $y\to x$ of the kernel $\Kc(y,x-y)$ are determined by the second fundamental form of the surface, or, what is equivalent, by the principal curvatures.

Finally, we take into account the structure properties of the symbol $\mF(x,\xi)$ as it depends on the principal and second symbols of the operator $\KF$, see Proposition \ref{Structure}. Since, at a given point, in C--co-ordinates, the entries in each summand of the   leading symbol $\kF_0$ do not depend on the geometry of $\G$ and the only other factor in this summand is linear in the principal curvatures, therefore each summand, and thus the whole symbol $\mF(x,\x)$, depends linearly on the principal curvatures of $\G$. This concludes the proof of Theorem \ref{Result}

\section{Eigenvalue asymptotics for the NP operator. symmetries and some reductions}\label{calculation}
In this section we apply the results on the spectrum of general polynomially compact pseudodifferential operators, obtained in Section \ref{nondeg}, to finding the asymptotics of eigenvalues of the NP operator. In our case, the dimension of the fiber equals 3, the polynomial $p(\o)$ has degree $3$, $p(\o)=\o(\o+\mathbb{k})(\o-\mathbb{k})=\o(\o^2-\mathbb{k}^2)$, so $\o_1=0, \o_2=\mathbb{k}, \o_3=-\mathbb{k}$.
In accordance with Theorem \ref{SymmAS}, in order to find the asymptotics of eigenvalues approaching the point $\o_\i$, $\i=1,2,3$, we should consider the operator $\MF=p(\KF)^2(\KF-\o_\i)^{-1}$ and find its principal symbol of order $-1$. The structure of this symbol has been just described in Section \ref{Sect.structure}. Finding the explicit expression for the terms $M^{(1)}_{\i}(\x_1,\x_2)$, $M^{(2)}_{\i}(\x_1,\x_2)$, for given Lame constants, requires some fairly tedious calculations, and this will be done in a later publication. We, however, can use some soft analysis to derive certain general properties of these functions, namely, we use the arbitrariness in the choice of the local co-ordinates, which is allowed by our construction, as well as the independence of matrix-functions $M_{\i}(\x)$ on the geometry of the surface $\G$.

The first property follows from the fact that the symbol $\kF(x,\x)$ should remain invariant as soon as we permute the co-ordinate axes $x_1$ and $x_2$. This will lead to the simultaneous permutations $\kb_1 \Leftrightarrow\kb_2$ and $\x_1\Leftrightarrow\x_2$.  Moreover, due to our choice of the frame in $\R^3$, the first two rows, as well as two first columns in $\mF_{\i}$ must interchange.   We denote by $V$ the transformation in $\R^3$ interchanging the first row with the second one,
 i.e., the matrix
 \begin{equation}\label{matrixV}
    V=
    \left(
      \begin{array}{ccc}
        0 & 1 & 0 \\
        1 & 0 & 0 \\
        0 & 0 & 1 \\
      \end{array}
    \right),
 \end{equation}
 $V=V^{-1}$

 Thus, by symmetry, the symbols $M^{(1)}_{\i}$ and $M^{(2)}_{\i} $ in \eqref{SymbolStructure} must interchange under these transformations, so we have
\begin{equation}\label{prop.1}
   M^{(2)}_{\i}(\x_1,\x_2)=V^{-1}M^{(1)}_{\i}(\x_2,\x_1)V.
\end{equation}
As a result, the representation \eqref{SymbolStructure} takes the form
\begin{equation}\label{prop.2}
   \mF_\i(x,\x)=\kb_1(x) M_{\i}(\x_1,\x_2)+\kb_2(x)V^{-1} M_{\i}(\x_2,\x_1)V,
\end{equation}
with just a single universal symbol $M_{\i}(\x)$.

Further on, if we change the direction of the $x_1$ axis to the opposite one, with the corresponding change of the sign of co-variable $\x_1$, we should get the same symbol,
i.e.,
\begin{equation*}
    \mF_{\i}(x,-\x_1,\x_2)=V_1^{-1}\mF_{\i}(x,\x_1,\x_2)V_1, \, V_1=\left(
             \begin{array}{ccc}
               -1 & 0 & 0 \\
               0 & 1 & 0\\
               0 & 0 & 1 \\
             \end{array}
           \right),
\end{equation*}
or, taking into account \eqref{SymbolStructure} and the arbitrariness of $\kb_1,\kb_2$ there, we obtain
\begin{equation}\label{change sign1}
    M_{\i}(-\x_1,\x_2) =V_1^{-1}M(\x_1,\x_2)V_1.
\end{equation}
In the same way, for the second co-ordinate,
\begin{equation}\label{change sign2}
    M_{\i}(\x_1,-\x_2) =V_2^{-1}M(\x_1,\x_2)V_2,\, V_2=\left(
             \begin{array}{ccc}
               1 & 0 & 0 \\
               0 & -1 & 0\\
               0 & 0 & 1 \\
             \end{array}
           \right).
\end{equation}
When we apply both \eqref{change sign1}, \eqref{change sign2}, we obtain
\begin{equation}\label{change sign3}
    M_{\i}(-\x)=V_{1,2}^{-1}M_{\i}(\x)V_{1,2}, \, V_{1,2}=\left(
             \begin{array}{ccc}
               -1 & 0 & 0 \\
               0 & -1 & 0\\
               0 & 0 & 1 \\
             \end{array}
           \right).
\end{equation}

Consider now the special case of $\G$ being the unit sphere  $S$ in $\R^3$.  All points on $\G$ are umbilical, moreover, $\kb_1(x)=\kb_2(x)=-1$ everywhere on $\G$. And here we can choose the local (orthogonal) co-ordinates in an arbitrary way Therefore, the symbol $\mF_\i(x,\x)$ equals here

\begin{equation}\label{prop.3}
    \mF_{\i}(x,\x)=-M_{\i}(\x_1,\x_2)- V^{-1}M_{\i}(\x_2,\x_1)V.
    \end{equation}

    For the sphere, the eigenvalues of the symbol $ \mF_{\i}(x,\x)$ are the same for all points $x$, and,  by symmetry, the integrand in \eqref{Formula4} is independent of the point $x$ and on the covector  $\x$. 
We recall now the explicit formulas for the eigenvalues of the NP operator on the sphere, see \eqref{ball}. This formulas show that the eigenvalues of $\KF$ are approaching each of three points of the essential spectrum from above. This,  by Theorem \ref{Result}, means that the integrand in the eigenvalue  asymptotic formula \eqref{Formula4} is always positive for $+$ sign and it is everywhere zero for the $-$ sign.  So, the eigenvalues of the matrix $\mF_{\i}(x,\x)$ in \eqref{prop.2} are nonnegative, with at least one of them being positive.

This implies, in particular, that from the point of view of our analysis in Sect.\ref{nondeg} we always have the nondegenerate case, and therefore, for any surface $\G$, there are infinitely many  eigenvalues of the NP operator approaching each of three points of the essential spectrum, with power-like asymptotics \eqref{Formula1}, \eqref{Formula2} with $d=2$ and at least one of quantities there, with $+$ or with $-$ sign, nonzero.

Now suppose that there exists at least one point $\xb$ on $\G$ where the principal curvatures are equal and  negative. Then the integrand with "+" sign in \eqref{Formula4} is positive at the point $\xb$,  and, by continuity, at all points at some neighborhood of $\xb$. Therefore, the expression in \eqref{integrand} is positive on this neighborhood. This means that the coefficient in the asymptotics of eigenvalues of the NP operator approaching $\o_\i$ from above is nonzero, therefore there are infinitely many such eigenvalues.

 This reasoning enables us to establish the same infiniteness for any \emph{strictly convex} body, such a convex body, where at any point on $\G$ at least one of principal curvatures is nonzero. In fact, it is known on $\G$ there must exist at least one umbilical point. At such point the principal curvatures are equal and  negative. Moreover, if at some point the principal curvatures are negative and sufficiently close to each other, and  the same reasoning as above establishes the  positivity of the asymptotic coefficient  in the formula for $\bl_k^{\i,+}$.

On the other hand, if at some point on $\G$ the body is concave, so that both principal curvatures are positive and equal, then the integrand in \eqref{Formula4} with $-$ sign is negative and therefore the coefficient in the asymptotics for $\bl_k^{\i,-}$ is nonzero, thus, there exist an infinite sequence approaching the corresponding point $\o_\i$ from below. By continuity, the same can be said about such eigenvalues for the body where the principal curvatures are positive and sufficiently close at some point.

\section{Conclusion}
The structure \eqref{prop.2} of the symbol hints on the economic way for calculating it. Since the symbols $M_{\i}^{(1)}$, $M_{\i}^{(1)}$ are universal and do not depend on the geometry of the surface, we can calculate the symbol $\mF_{\i}$ for a surface where one of the principal curvatures vanishes, i.e., for a cylindrical surface.

Nevertheless, even here this calculation is fairly lengthy and tedious,  and we postpone it to another publication. We just mention here that the crucial point is the explicit expression for the first and second terms in the symbol of the parametrix of the Lam\'e operator in curved co-ordinates, obtained in \cite{Kozh1}, \cite{Kozh2}.

 We just  list the resulting properties whose proof will be presented later.
\begin{enumerate}
    \item For any smooth bounded body there always exist infinitely many eigenvalues of the NP operator, approaching each point of the essential spectrum from above.
    \item If there exists at least one point on $\G$ where the body is strictly concave, i.e., both principal curvatures are nonnegative, while at least one is positive, then there exist infinitely many eigenvalues of the NP operator approaching the points of the essential spectrum from below.
    \item The leading term in the asymptotics of eigenvalues approaching $0$ does not depend on Lam\'e parameters.
\end{enumerate}

The author is grateful to Y.Miyanishi for enlightening discussions.


\begin{thebibliography}{99}
\bibitem{Adams} M. Adams, \emph{Spectral properties of zero order pseudodifferential operators.}
 {J. Funct. Anal.} \textbf{32} (3), 1983, 420--441.
\bibitem{AgrLame}M. Agranovich, B. Amosov, M. Levitin. \emph{Spectral problems for the Lam\'e system with spectral parameter
 in boundary conditions on smooth or nonsmooth boundary.} {Russian J. Math. Phys.} \textbf{6}, (5), 1999, 247--281.


\bibitem{Ammari}H. Ammari, G. Ciraolo, H. Kang, H. Lee, G. Milton.\emph{ Spectral theory of a Neumann--Poincar\'{e} operator and analysis of cloaking due to anomalous localized resonance,} {Arch. Rat. Mech. Anal.}, \textbf{208}, 2013, 667--692.
\bibitem{3D}K. Ando, H. Kang, Y. Miyanishi,\emph{ Elastic Neumann--Poincar\'e operators in three dimensional smooth domains:
 polynomial compactness and spectral structure,} {Int. Math. Res. Notes}, \textbf{2019} (12), 2019, 3883--3900.


   \bibitem{IOP} K. Ando, H. Kang, Y. Miyanishi.\emph{ Spectral structure of elastic Neumann--Poincar\'e operators.}
  {J.Physics: Conference series}, \textbf{965}, 2018, 012027.
   \bibitem{Ando2D}  K.Ando; Y.--G. Ji; H. Kang ; K. Kim; S. Yu. \emph{Spectral properties of the Neumann--Poincar\'e
       operator and cloaking by anomalous localized resonance for the elasto-static system.} European J. Appl. Math. 29
        (2018), no. 2, 189--225.


\bibitem{AKM}K. Ando, H. Kang, Y. Miyanishi. \emph{Convergence rate for eigenvalues of the elastic Neumann--Poincar\'e operator
on smooth and real analytic boundaries in two dimensions.} \texttt{arXiv:1903.07084}
\bibitem{AKM2}K. Ando, H. Kang, Y. Miyanishi.\emph{ Spectral analysis of Neumann-Poincare operator.}\texttt{ arXiv:2003.14387 }

\bibitem{BS obzor}Birman, M. Sh.; Solomyak, M. Z.\emph{ Asymptotic properties of the spectrum of differential equations.} (Russian) Mathematical analysis, 14 (Russian), pp. 5--58.  Akad. Nauk SSSR Vsesojuz.Inst. Nauchn. i Tehn. Informacii, Moscow, 1977. English translation in: J. Sov. Math. 12 (1979), 247--283.
\bibitem{BS}M.Birman, M.Solomyak.\emph{ Asymptotic behavior of the spectrum of pseudodifferential operators with anisotropically
 homogeneous symbols.} (Russian) Vestnik Leningrad. Univ. 1977, no. 13 Mat. Meh. Astronom. vyp. 3, 13--21. English translation in:
 Vestnik Leningr. Univ. Math. 10 (1982) 237--247.

 \bibitem{CdV1}Y. Colin de Verdi\`{e}re, L. Saint--Raymont. \emph{Attractors for two dimensional waves with homogeneous Hamiltonians of
 degree 0.}  Comm. Pure Appl. Math. 73 (2020), no. 2, 421--462.\texttt{arXiv: 1801.05582}
\bibitem{CdV2}Y. Colin de Verdi\`{e}re, \emph{Spectral theory of pseudo--differential operators
of degree 0
and application to forced linear waves.} \texttt{arXiv: 1804.03367v2}(to appear in Analysis $\&$ PDE)
\bibitem{DLL}Y. Deng, H. Li, H. Liu. \emph{On spectral properties of Neuman--Poincar\'e operator and plasmonic resonances in 3D elastostatics,}
 \emph{J. Spectral Theory}, \textbf{9} (2019), no. 3, 767--789.

\bibitem{Duduchava} R. Duduchava, D. Natroshvili. \emph{Mixed crack type problem in anisotropic elasticity.} Math. Nachrichten, \textbf{191}
(1998) 83--107.
\bibitem {Duduchava2} R. Duduchava. \emph{The Green formula and layer potentials.} \emph{Integr. Equat. Oper. Theory}, \textbf{41}, 2001, 127--178.



\bibitem{Zwor} S. Dyatlov, M. Zworski. \emph{Microlocal analysis of forced waves}. Pure Appl. Anal. 1 (2019), no. 3, 359--384. \texttt{arXiv:1806.00809}.
 \bibitem{GalkoZwor} G.Galkowski, M.Zworski. \emph{Viscosity limits for 0th order pseudodifferential operators.} \texttt{arXiv:1912.09840 }
 \bibitem{Grubb84} G.Grubb. \emph{Singular Green operators and their spectral asymptotics.} Duke Math. J. \textbf{51} (1984), 477--528.
 \bibitem{Hormander} L. H\"ormander. \emph{The spectral function of an elliptic operator.} Acta Math. 121 (1968), 193--218.
\bibitem{KK}H. Kang, D. Kawagoe. \emph{Surface Riesz transforms and spectral properties of elastic Neumann--Poincar\'e operators on less
 smooth domains in three dimensions.} \texttt{arXiv:1806.02026}.
    \bibitem{Kozh1} A. Kozhevnikov. \emph{The basic boundary problems of the static elasticity theory and their Cosserat spectrum}.
    Math. Zeitschrift, \textbf{213} (1993),  241-270.
    \bibitem{Kozh2} A. Kozhevnikov, T. Skubachevskaya. \emph{Some applications of pseudo-diffeential operators to elasticity.}
    Hokkaido Math. Journ., \textbf{26} (1997), 297-322.
   \bibitem{KuprPot}V. Kupradze.\emph{ Potential Methods in the Theory of Elasticity.} Gosudarstv. Izdat. Fiz.-Mat. Lit., Moscow, 1963. (Russian), Daniel Davey, 1965. (English.)

\bibitem{Kupr79}V. Kupradze, T. Gegelia, T. Bashelishvili, T. Burchaladze. \emph{Three-Dimensional Problems in the Mathematical
Theory of Elasticity and Termoelasticity }. Moscow: Nauka, 1976. (Russian); North Holland, 1979.(English)
\bibitem{Landkof}N.S. Landkof N.S. \emph{Foundations of modern potential theory.}"Nauka'', Moscow 1966. (Russian.) Springer, 1972. (English.)
\bibitem{LiLiu}H. Li, H. Liu. \emph{On three--dimensional plasmon resonance in elastostatics.} \emph{Ann. Mat. Pura Appl}.
 (4)\textbf{ 196} (3), 2017, 1113--1135.
\bibitem{LiuLame}G.Liu, \emph{ Determination of isometric real-analytic metric and spectral
invariants for elastic Dirichlet-to-Neumann map on Riemannian
manifolds.}\texttt{ arXiv:1908.05096}, 2019.
\bibitem{M} Y. Miyanishi. \emph{{Weyl's law for the eigenvalues of the Neumann--Poincare operators in three dimensions: Willmore energy and surface geometry.}}  \texttt{arXiv:1806.03657}
 \bibitem{MR} Y. Miyanishi, G. Rozenblum,
\emph{Eigenvalues of the Neumann--Poincar\'e operator in dimension 3: Weyl's law and geometry.}
Algebra i Analiz 31 (2019), no. 2, 248--268.
\bibitem{MR3D}Y. Miyanishi, G. Rozenblum, \emph{Spectral properties of the Neumann-Poincar\'e operator
 in 3D elasticity.} \texttt{arXiv:1904.09449}, IMRN, \texttt{https://doi.org/10.1093/imrn/rnz341}, 2020.


\bibitem{Tao} Z.Tao. \emph{0-th Order pseudo-differential operator on the circle.}\texttt{arXiv:1909.06316}



\bibitem{Wang}J. Wang. \emph{The scattering matrix for 0 th order
pseudodifferential operators.} \texttt{arXiv:1909.06484}
\bibitem{Yaf}D. Yafaev.
\emph{Scattering by magnetic fields.} (Russian. )
Algebra i Analiz  \textbf{17}  (2005),  no. 5, 244--272;  translation in
St. Petersburg Math. J.  \textbf{17}  (2006),  no. 5, 875--895
\end{thebibliography}
\end{document}